\documentclass[reqno,12pt]{amsart}
\usepackage{latexsym,graphicx,epsfig,graphics}
\usepackage{amsmath, amsthm, amssymb}
\usepackage{pgf,tikz}
\usepackage{url}
\usepackage{verbatim}
\usepackage{enumerate}
\usepackage{rotating}
\usepackage{array}
\usepackage{longtable}
\usepackage{calc}
\usepackage{multirow}
\usepackage{hhline}
\usepackage{ifthen}
\usepackage{amsthm}
\usepackage{mathrsfs}
\usepackage{amsfonts}
\usepackage{eucal}

\usetikzlibrary{arrows,topaths,calc}
\definecolor{sqsqsq}{rgb}{0.13,0.13,0.13}
\theoremstyle{plain}
\newtheorem{theorem}{Theorem}[section]
\newtheorem{proposition}[theorem]{Proposition}
\newtheorem{corollary}[theorem]{Corollary}
\newtheorem{lemma}[theorem]{Lemma}

\theoremstyle{definition}
\newtheorem{definition}[theorem]{Definition}

\newtheorem{example}[theorem]{Example}
\newtheorem{conjecture}[theorem]{Conjecture}

\numberwithin{equation}{section}
\title{An introduction to the mincut graph}

\author{C. Kriel$^{\S}$ \and E. Mphako-Banda$^{\S}$ }
\begin{document}

\date{}
\maketitle
 
\centerline{$^\S$ School of Mathematics,}
\centerline{University of the Witwatersrand, PB 3, WITS, 2050,
RSA.}
\centerline{ Email: christo.kriel@wits.ac.za; eunice.mphako-banda@wits.ac.za}

\begin{abstract}
There is a well-documented research programme on graph operators which addresses questions such as `Which graphs appear as images of graphs?'; `Which graphs are fixed under the operator?'; `What happens if the operator is iterated?' In this paper, we introduce an intersection graph called the mincut graph and state its characteristics. We address some of the research programme questions for the mincut graph operator and conclude by stating further research questions on the mincut graph.

\par
\vspace{.1in}
\par
\vspace{.1in}

\end{abstract}

\section{Introduction}
\label{sec:intro}
Given a set $S$ and a family $F=\{S_1, S_2 \ldots , S_i\}$ of subsets of $S$, an intersection graph of $F$ is a graph with vertices $v_i$ corresponding to each of the $S_i$ and two vertices $v_i$ and $v_j$ are adjacent if $S_i\cap S_j \neq \emptyset$, see \cite{graphreps, grtheory}. In 1945 Szpilrajn-Marczewski proved that every graph is an intersection graph, \cite{graphreps}. One of the first class of intersection graphs to be widely studied was the line graph, generalised as $(X,Y)$-intersection graphs in \cite{linegrgen}, while in the $1970$'s chordal graphs were first characterised in terms of intersection graphs. Other intersection graphs that are studied intensively are interval and circular-arc graphs, competition graphs, $p$-intersection and tolerance graphs, to name but a few, see \cite{algogr,intersectiongr,intersecintro}. Problems involving intersection graphs often have real world applications in topics like biology, computing, matrix analysis and statistics, see \cite{algogr,intersectiongr}.

In this paper, we introduce an intersection graph of a graph $G$, called a \emph{mincut graph}, with vertex set the minimum edge-cuts of $G$ such that two vertices in the intersection graph are adjacent if their corresponding minimum edge-cut sets have non-empty intersection. We then study some of its properties and characteristics. In doing so we follow the research programme on graph operators, as introduced by Prisner in the 1995 monograph ``Graph Dynamics'', see \cite{grdynamics}.

As the line graph of a graph $G$ reflects the mutual positions of the edges, see \cite{prissurv}, so the mincut graph reflects the mutual positions of the minimum edge-cuts. We believe that this new structure can be used in examining the effect of certain graph parameters on connectivity and possibly also lead to the identification of some new connectivity parameters. For example, it is known that the maximum number of mincuts has different order of magnitude for graphs with odd and even edge connectivity, see \cite{kanevsky}. Let $G$ be a simple graph with minimum edge-cut number $\lambda$ and $n$ vertices, with $X=\{X_1,\, X_2, \, \ldots\, ,X_i\}$ the set of mincuts of $G$, then we have the following upper bounds for $|X|$, the number of minimum cuts of $G$, see  \cite{lehel}:
\begin{equation*}
	|X|\leq\begin{cases}
		\frac{2n^2}{(\lambda+1)^2}+\frac{(\lambda-1)n}{\lambda+1}, & \text{if $\lambda \geq 4$ and $\lambda$ is an even integer},\\
		(1+\frac{4}{\lambda+5})n, & \text{if $\lambda>5$ and $\lambda$ is an odd integer}.
	\end{cases}
\end{equation*}

The problem of counting the number of mincuts of a graph is well known in the literature and various data structures have been created  to represent all these mincuts, see for example \cite{dinic,fleischer,gabow}. In the mincut graph, the bounds on $|X|$ become bounds on the order of the mincut graph, since every minimum edge-cut is a vertex in the mincut graph. The effect on these upper bounds of certain properties of graphs such as \emph{radius and diameter} or \emph{maximum and minimum degree} are investigated in \cite{chandran}. Although outside the scope of this paper it would be interesting to know whether these upper bounds are related to or have similar effects on parameters in the mincut graph.

\section{Introducing the mincut graph}
\label{sec:mincutgraph}

\begin{definition}
Let $G$ be a simple connected graph, then an edge-cut of $G$ is a subset $X$ of $E(G)$, such that $G-X$ is disconnected. An edge-cut of minimum cardinality in $G$ is a \emph{minimum edge-cut} and its cardinality is the edge-connectivity of $G$, denoted $\lambda(G)$. We will call such a minimum edge-cut a \emph{mincut} of $G$. We call a mincut that disconnects a single vertex from the graph a \emph{trivial cut}.
\end{definition}

\begin{example} We illustrate this concept with an example. The diagrams in Figure \ref{fig:egmincut} are the wheel graph, $W_6$, and the Peterson graph, both with minimal edge cuts labeled $\{e_1 \, , \ldots \, , \, e_5 \}$. However, $\{e_1 \, , \ldots \, , \, e_5 \}$ is not a mincut for either of the graphs, since in each of the graphs the removal of any of the three edges incident on a single vertex of degree $3$  will also give a disconnected graph. 

\begin{figure}[!h]
	\begin{center}
		\includegraphics[scale=1.5]{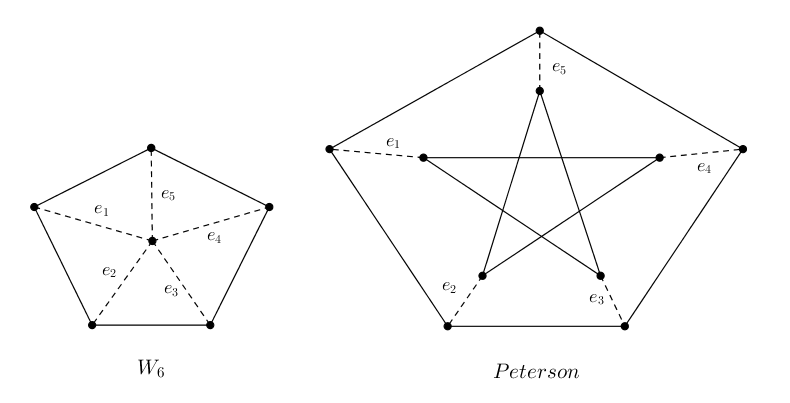}
\end{center}
	\caption{The edge sets $\{e_1 \, , \ldots \, , \, e_5 \}$ are minimal, but not mincuts. }
	\label{fig:egmincut}
\end{figure}
\label{eg:mincut}

\end{example}

\begin{definition}
Let $X=\{X_1, \, X_2, \, \ldots \, X_i\}$ be the set of all mincuts of a simple connected graph $G$. Represent each of the $X_i$ with a vertex $v_i$ such that two vertices $v_i$ and $v_j$ are adjacent if $X_i\cap X_j \neq \emptyset$, and call this intersection graph the \emph{mincut graph} of $G$, denoted by $X(G)$.
\label{def:XG}
\end{definition}

\begin{lemma}
	Let $G$ be a disconnected graph, then $X(G)$ is the null graph, $K_0$, with no vertices or edges.
	\label{lem:nullgraph}
\end{lemma}
\begin{proof}
	Since $G$ is disconnected, $\lambda(G)=0$ and $G$ has no edge-cut. Hence $X(G)$ has no vertices and consequently no edges.
\end{proof}

\begin{theorem}
	Let $G$ be a connected graph with $n$ vertices and minimum edge-cut number $\lambda (G)=k, \,k\geq 1$. Then $X(G)$, the mincut graph of $G$, is unique.
	\label{th:XGunique}
\end{theorem}
\begin{proof}
	Since $k \geq 1$ we have to remove at least one edge to disconnect the graph. Thus we know there is at least one edge-cut and hence at least one mincut. Thus, by Definition \ref{def:XG}, $X(G)$ exists and has at least one vertex.
	
	Let $X=\{X_1, \, X_2, \, \ldots \, X_i\}$ be the set of all mincuts of a simple connected graph $G$. Since $k \geq 1$ we know that $X$ is non-empty. Let $X'(G)$ and $X''(G)$ be two mincut graphs of $G$. By the definition of a mincut graph we have $V(X'(G))=V(X''(G))$ corresponding to $X$. Since two vertices in the mincut graph of $G$ are adjacent if their corresponding mincuts have non-empty intersection, we must have $v_iv_j\in E(X'(G))$ if and only if $v_iv_j\in E(X''(G))$ for all $v_i, \,v_j \in V(X'(G))$ and $v_i, \,v_j \in V(X''(G))$. Therefore $X'(G)\cong X''(G)$ and we conclude that $X(G)$ is unique.
\end{proof}

We now discuss the mincut graphs of some well-known families of graphs. Proofs of the constructions are straightforward and we merely outline them here.

For a tree $T_n$ every edge is a bridge and $\lambda(T_n)=1$. Each of the $n-1$ edges of $T_n$ is a mincut and hence $X(T_n)$ has $n-1$ vertices. But none of the singleton mincuts intersect and thus $X(T_n)\cong \overline{K_{n-1}}$, the empty graph (no edges) on $n-1$ vertices.

For the complete bipartite graph $K_{m,n}$ with vertex partition sets of size $m$ and $n$, respectively, such that $n>m>1$, $X(K_{m,n})\cong \overline{K_n}$. This follows from the fact that the mincuts of $K_{m,n}$ are exactly the $m$ edges incident on each of the $n$ vertices of degree $m$ in the larger of the two vertex partitions. Since $K_{m,n}$ is bipartite none of these intersect and we have a mincut graph with $n$ vertices and no edges. If $m=1$ we simply have a tree.

Let $W_n$ be the wheel on $n$ vertices, $n>4$, then, for any $W_n$, $\lambda(W_n)=3$ and the mincuts are exactly the trivial cuts with edges incident on every vertex $v_i$ on the ``rim'' of the wheel. By labeling the $n-1$ vertices on the rim in sequence, we see that every mincut $X_i$ has non-empty intersection with $X_{i-1}$ and $X_{i+1}$, thus giving $X(W_n)\cong C_{n-1}$.

Recall that the line graph, $L(G)$, of a graph $G$ has the edges of $G$ as its vertices, such that two vertices in $L(G)$ are adjacent if their corresponding edges in $G$ have a vertex in common.

For $C_n$, the cycle on $n$ vertices, $X(C_n)\cong L(K_n)$, the line graph the complete graph on $n$ vertices. Given that the edge connectivity of $C_n$ is $\lambda(C_n)=2$, and any choice of two edges is a mincut, the set of all mincuts, $X$, of $C_n$ is the set of all two element subsets of $E(C_n)$, where $|E(C_n)|=n$. But the vertex set of $L(K_n)$ is the set of all two element subsets of the $n$-element vertex set of $K_n$, the edges of $K_n$, and the proof follows.

\begin{example}
	In Figure \ref{fig:mincutgraphs}, we show some examples of the mincut graphs of some of the graphs dealt with in the preceding section.
\end{example}

\begin{figure}[!h]
	\begin{center}
		\includegraphics[width=\textwidth]{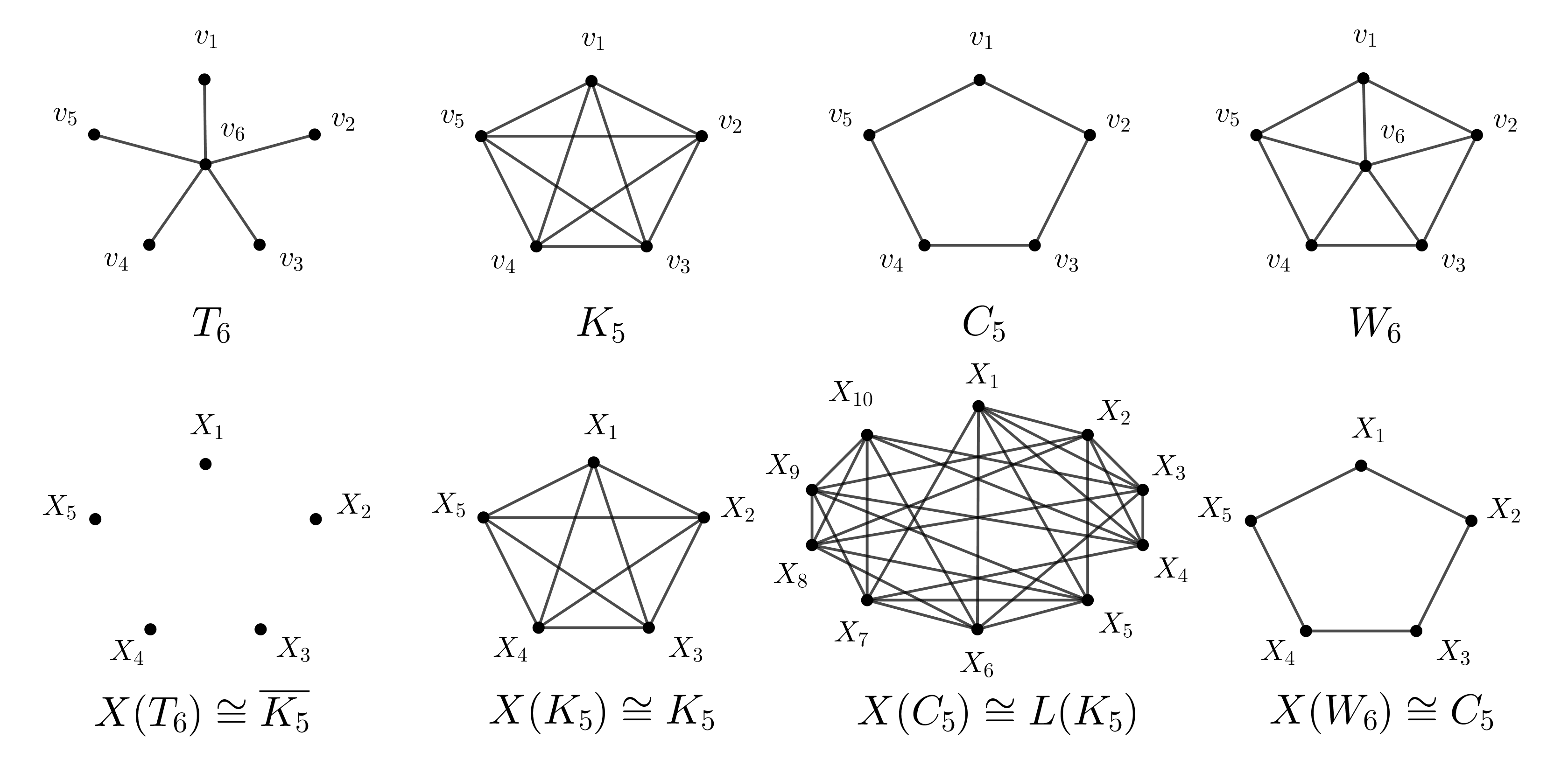}
	\end{center}
	\caption{Mincut graphs of some well-known classes of graphs.}
	\label{fig:mincutgraphs}
\end{figure}

\section{Super edge-connected graphs}
\label{sec:superedge}
In this section, we define super-edge connected, or super-$\lambda$ graphs and give a sufficient condition for a mincut graph of a graph $G$ to be isomorphic to $G$.

\begin{definition}
A graph $G$ is \emph{maximally edge connected} when $\lambda=\delta$, where $\lambda$ is the cardinality of a minimum edge-cut and $\delta$ is the minimum vertex degree of $G$.
\label{def:maxedge}
\end{definition}

\begin{definition}
A  graph $G$ is \emph{super-$\lambda$} if every minimum edge-cut set is \emph{trivial}; that is, consists of the edges incident on a vertex of minimum degree.
\label{def:superlambda}
\end{definition}

The following proposition states three of five sufficient conditions for a graph to be super-$\lambda$, see \cite{handbookgraphth}. The first two conditions were proved by Lesniak in 1974, see \cite{lessuperl}.

\begin{proposition}
	\label{prop:suffsuperl}
	Let $G$ be a graph of order $n$, minimum degree $\delta$ and maximum degree $\Delta$. Then $G$ is super-$\lambda$ if any of the following conditions are satisfied:
	\begin{enumerate}
		\item Let $G\neq K_{n/2}\Box K_2$, the cartesian product of $K_{n/2}$ and $K_2$. If for any non-adjacent $u,v\in V(G)$, $\textrm{deg}u+\textrm{deg}v\geq n,$	then $G$ is super-$\lambda$.
		
		\item If for any non-adjacent $u,v\in V(G)$, $\textrm{deg}u+\textrm{deg}v\geq n+1,$	then $G$ is super-$\lambda$.
		
		\item If $\delta \geq \lfloor \frac{n}{2} \rfloor +1$, then $G$ is super-$\lambda$.
		
		
	\end{enumerate}
	
\end{proposition}

\begin{proposition}
If $G$ is $r$-regular and super-$\lambda$, then $X(G)\cong G$.
\label{prop:X(G)=G}
\end{proposition}

\begin{proof}
Since $G$ is super-$\lambda$ the mincuts of $G$ are the edges incident on vertices of minimum degree and since $G$ is regular this is true for every vertex. Hence, there is a one-to-one correspondence between the vertices of $X(G)$ and the vertices of $G$ and the adjacencies are preserved. 
\end{proof}

\begin{corollary} Let $K_n$ be the complete graph on $n$ vertices, $K_{n,n}$ the complete bipartite graph with equal vertex partitions and $L(K_n)$ the line graph of the complete graph. If $n>2$, then
\begin{enumerate}
\item $X(K_n)\cong K_n$
\item $X(K_{n,n})\cong K_{n,n}$
\item $X(L(K_n))\cong L(K_n).$
\end{enumerate}
\label{corol:Xsuperl}
\end{corollary}
\begin{proof}
If $n=2$, $K_2$ is a tree, $K_{2,2}\cong C_4$ and $L(K_2)\cong K_1$. If $n>2$ all three graphs are clearly regular and super-$\lambda$.
\end{proof}

\section{When is a graph a mincut graph}

In this section, we address the question of which graphs are mincut graphs, that is, which graphs lie in the image of the mincut operator, which we will refer to as the $X$-operator and denote as $X(\cdot)$. We show that every graph is a mincut graph, by constructing a super-graph $G^*$ from a given graph $G$ such that $G$ is the mincut graph of $G^*$. Furthermore, we show that every graph has an infinite number of $X$-roots. For more on the concepts and definitions on graph operators used here see Prisner in \cite{grdynamics}.

\begin{definition}
	A graph operator is a mapping $\phi$ which maps every graph $G$ from some class of graphs to a new graph $\phi(G)$.
	\label{def:graphop}
\end{definition}

We define the operator recursively, such that, if $\phi$ is an operator and $k$ a positive integer, then $\phi^1=\phi$ and $\phi^k(G)=\phi(\phi^{k-1}(G))$, for $k\geq 2$, see \cite{grdynamics,prissurv}. Since an operator, $\phi$, is a mapping on the set of finite graphs into itself, it is natural to ask which graphs lie in the image of $\phi$ and, given $\phi(G)$, are we able to determine $G$?

Every graph is an intersection graph, see \cite{graphreps, grtheory}. We show that every graph is the mincut graph of not just one graph, but more. It is possible to construct an infinite family of non-isomorphic graphs that map to the same mincut graph.

\begin{lemma} Let $G$ be a super-$\lambda$ graph with $H\subseteq G$ the induced subgraph on the set of vertices of $G$ such that $V(H)=\{v\in V(G)| deg(v)=\delta(G) \}$, then $H\cong X(G)$.
	\label{lem:XisoHsupersub}
\end{lemma}
\begin{proof}
	Let $x_i \in V(X(G))$ be the vertex in $X(G)$ corresponding to the mincut $X_i\in X$, the set of all mincuts in $G$. The mincuts in $G$ are exactly the incident edge-sets on the vertices $v\in V(H)$ and hence $|X|=|V(H)|$. Thus for every $v_i\in V(H)$ there is exactly one $x_i \in V(X(G))$. Furthermore, if $v_iv_j\in E(H)$ then $X_i\cap X_j \neq \emptyset$, where $X_i$ and $X_j$ are the incident edge sets on $v_i$ and $v_j$. If $v_iv_j \notin E(H)$ then $X_i \cap X_j = \emptyset$, since they are the incendent edge-sets on $v_i$ and $v_j$. Hence $x_ix_j \in E(X(G))$ if and only if $v_iv_j \in E(H)$ and isomorphism follows.
\end{proof}

\begin{example}
	The graph $G$ in Figure \ref{fig:lemmaeg} is super-$\lambda$, by condition $1$ of Proposition \ref{prop:suffsuperl} with vertices of minimum degree $v_1, \, v_2, \, v_3, \textrm{ and } v_4$ and $X(G) \cong K_{1,3}$ is the induced subgraph on vertices $v_1 \ldots v_4$.
	
	\begin{figure}[!h]
		\begin{center}
			\includegraphics{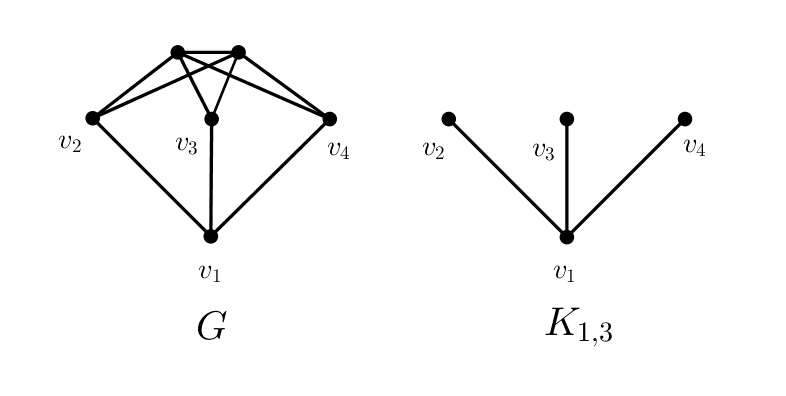}
		\end{center}
		\caption{ }
		\label{fig:lemmaeg}
	\end{figure}
	
\end{example}

\begin{corollary}
	If $G$ is super-$\lambda$ and $r$-regular, $G\ncong K_2$, then $X(G)\cong G$.
	\label{corol:XGisGifsuper-r}
\end{corollary}
\begin{proof}
	 Let $H\subseteq G$ be the induced subgraph on the set of vertices of $G$ such that $V(H)=\{v\in V(G)| deg(v)=\delta(G) \}$. Since $G$ is $r$-regular and super-$\lambda$, ${V(H)=V(G)}$ and the edge-set incident on \emph{every} vertex is a mincut. Hence, if $v_iv_j \in E(G)$, then $X_i\cap X_j \neq \emptyset$ and $x_ix_j \in E(X(G))$. We exclude $K_2$ since $K_2\cong T_2$ and $X(T_2)\cong K_1$ as we see in Section \ref{sec:mincutgraph}.
\end{proof}

\begin{definition}
	The \emph{join} $G=G_1 \vee G_2$ of two graphs $G_1$ and $G_2$ has vertex set $V(G_1)\cup V(G_2)$ and edge set $E(G_1)\cup E(G_2)\cup \{uv|u \in V(G_1), \, v \in V(G_2)\}$.
	\label{def:join}
\end{definition}

In other words, the join of two graphs joins every vertex in the one graph to every vertex in the other graph. Thus, for example, we can say $K_n=K_1 \vee K_{n-1}$ and $K_{m,n}=\overline{K_m} \vee \overline{K_n}$. If one of the graphs is $K_1$ the operation is also called the \emph{vertex join} of the graph.

\begin{lemma}
	Let $G$ be an $r$-regular graph of order $n>r+1$. Let $G^*=K_1 \vee G$, then $X(G^*)\cong G$.
	\label{lem:joinregular}
\end{lemma}

\begin{proof}
	The join $K_1\vee G$ yields a graph with one vertex of degree $deg(v)>r+1$ connected to every other vertex which is now of degree $r+1$. If $G$ is super-$\lambda$ then $G^*$ is super-$\lambda$, since we have added one vertex of higher degree than all the others, and $G^*[V(G)]$ is the induced subgraph on vertices of degree $\delta$. Then, by Lemma \ref{lem:XisoHsupersub}, $X(G^*)\cong G$.
	Suppose $G$ is not super-$\lambda$ and $X$ is a non-trivial mincut of $G$ with $\langle Y, \, \bar{Y} \rangle$ the two components of $G-X$, then in $G^*$ every vertex in $Y$ and $\bar{Y}$ is adjacent to the added vertex and hence $G-X$ is no longer disconnected. Hence $X$ is no longer a mincut and we conclude that $G^*$ is super-$\lambda$ with $G^*[V(G)]$ the induced subgraph on vertices of degree $\delta$ and by Lemma \ref{lem:XisoHsupersub}, $X(G^*)\cong G$.
\end{proof}

\begin{theorem} Every graph $G$ is the mincut graph of an infinite family of graphs.
	\label{th:AllGXG}
\end{theorem} 
\begin{proof}
	Let $G$ be a graph of order $p$ with minimum vertex degree $\delta(G)$ and maximum vertex degree $\Delta(G)$. We construct a super-$\lambda$ super graph $G^*$ such that $G\subset G^*$ is the induced subgraph on the vertices of $G^*$ of minimum degree, that is, such that $V(G)=\{v\in V(G^*)| deg(v)=\delta(G^*)\}$. Then, by Lemma \ref{lem:XisoHsupersub}, $G\cong X(G^*)$.
	
	Suppose $G$ is not $r$-regular. We start the construction with $G\cup K_m, \, m=\Delta(G)-\delta(G)$, the disjoint union of $G$ and a complete graph $K_m$. The choice of $m$ is the minimum number of vertices necessary to add edges between the two graphs so that $deg_{G^*}(v)=\Delta(G)$ for all $v\in V(G)$. We add edges between the graphs and increase $m$ where necessary until the following three conditions are met:
	\begin{enumerate}
		\item $deg_{G^*}(v)=k$ for all $v\in V(G)$,
		\item $deg_{G^*}(v)>k$ for all $v\in K_m$,
		\item $deg(u)+deg(v)\geq n+1$ for all non-adjacent vertices $u,v \in V(G^*)$, where $|V(G^*)|=n$.
	\end{enumerate}
	
	Then $G^*$ is super-$\lambda$ by condition $2$ of Proposition \ref{prop:suffsuperl} and $G$ is the induced subgraph on vertices of minimum degree in $G^*$. That is, $V(G)=\{v\in V(G^*)| deg(v)=\delta(G^*)\}$ and $\lambda(G^*)=deg_{G^*}(v), \,v\in V(G)$. Hence, by Lemma \ref{lem:XisoHsupersub} $G\cong X(G^*)$. We note that, by condition 1 of Proposition \ref{prop:suffsuperl}, if $G^*\ncong K_{\frac{n}{2}}\Box K_2$ we can have $deg(u)+deg(v)\geq n$ for all non-adjacent vertices $u$ and $v$.
	
	If $k+k<n+1$ add edges between $G$ and $K_m$ until $2k\geq n+1$. This ensures that condition $3$ is met for all non-adjacent $u,v\in V(G)$. If $k+m-1<n+1$ increase $k$, and $m$ if necessary, until we know that the graph is super-$\lambda$. Then increase $m$ to ensure that condition $2$ eventually applies, as clarified in Example \ref{eg:pawX}.  

	If $G$ is $r$-regular of order $p>r+1$, then our construction simply needs $G^*=K_1 \vee G$, and $X(G^*)\cong G$, by Lemma \ref{lem:joinregular}. 
	
	If $G$ is $r$-regular, but $p\leq r+1$ then $p-1\leq r$. But the maximum degree of a vertex in any simple graph of order $p$ is $p-1$ so $G$ must be complete. We start with $m=2$ and join all the vertices of $G$ to $K_m$. Then simply increase $m$ until $m>p+2$ without adding any new edges between $G$ and $K_m$. Both $G$ and $K_m$ are super-$\lambda$ and we have added $2p$, more than $\delta(G^*)$, edges between them.
	
	If $G$ is not connected we follow the same procedure, connecting each component of $G$ to $K_m$.
	
	Once we have $G^*$ super-$\lambda$ with $V(G)=\{v\in V(G^*)| deg(v)=\delta(G^*)\}$ we can keep increasing $m$ and still have a graph that fulfils the conditions of Lemma \ref{lem:XisoHsupersub} since we are only adding trivial edge cuts but they are not mincuts. Thus for any graph $G$ we can construct an infinite family of graphs $G^*$ such that $X(G^*)\cong G$.
\end{proof}


\begin{example}
	We illustrate the construction of an infinite family of graphs $\mathcal{F}(G^*)$ such that $X(G^*)\cong K_{1,3}$. Please refer to Figure \ref{fig:starX}.
	\label{eg:starX}
	We have $\Delta(K_{1,3})=3, \, \delta(K_{1,3})=1$ and the order of $K_{1,3}$ is $p=4$. We let $m=3-1=2$ and add edges between $K_{1,3}$ and $K_2$ until the vertices $\{v_1 \ldots v_4 \}$ are of degree $3$, as shown in graph $G^*_1$ in Figure \ref{fig:starX}. Now we have $n=6$ and $k+k=6$. However, since $G^*_1\neq K_3\Box K_2$, the Cartesian product of the two graphs, condition $1$ of Proposition \ref{prop:suffsuperl} tells us that it is sufficient for $deg(u)+deg(v)\geq n$ for $G^*$ to be super-$\lambda$ in this case. Hence, by now increasing $m$, for any $m\geq 5$ as in $G^*_2$ in the figure, $G^*$ would be super-$\lambda$, since the minimum edge cuts in a complete graph are only the trivial cuts and $deg_{G^*}(v)=3$ for all $v\in V(G)$.
	\begin{figure}[!h]
		\begin{center}
			\includegraphics{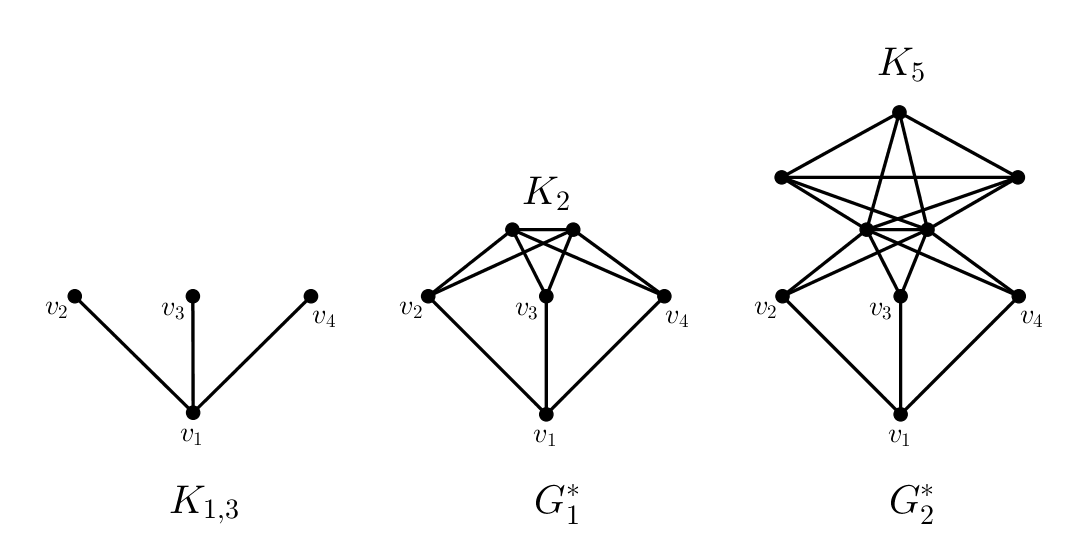}
		\end{center}
		\caption{Graphs $G^*_1$ and $G^*_2$, constructed such that $X(G^*_1)\cong X(G^*_2) \cong K_{1,3}$.}
		\label{fig:starX}
	\end{figure}
\end{example}

In the following example, we demonstrate a case where it is necessary to increase $k$ and consequently $m$ to ensure that $G^*$ is super-$\lambda$ such that $V(G)=\{v\in V(G^*)| deg(v)=\delta(G^*)\}$.

\begin{example}The graph $G$, in Figure \ref{fig:pawX}, formed by the one point join of $C_3$ and $P_2$, is sometimes called ``the Paw''. We have $\Delta(G)=3, \, \delta(G)=1$ and $n=4$. We proceed as in Example \ref{eg:starX} and let $m=2$ and add edges until $deg_{G^*}(v)=3$ for all $v\in V(G)$. We obtain the graph $G^*_1$ in Figure \ref{fig:pawX} with $n=6$ and $deg(u)+deg(v)=6$ for any two non-adjacent vertices. Condition 1 of Proposition \ref{prop:suffsuperl} does not apply here, however, since $G^*_1\cong K_3 \Box K_2$ and in fact the edge set $\{e_1,\,e_2,\, e_3\}$ is a non-trivial mincut. Hence we increase $k$, the degree of each vertex of $G$ in $G^*$, by one in $G^*_2$ which also necessitates increasing $m$ by one. Now $G^*_2$ is super-$\lambda$ since $deg(u)+deg(v)=8$ for any two non-adjacent vertices and $n=7$. The problem remains that one of the vertices in $K_3$ also has degree $4$ and hence is also a mincut, implying $X(G^*)\ncong G$. We increase $m$ to $k+2$ in $G^*_3$ which means that all vertices of $K_m$ in $G^*$ now have degree at least $5$. Also we have added only trivial cuts, since $K_6$ is complete and so we still have a super-$\lambda$ graph with vertices of minimum degree exactly the vertices of $G$. Clearly increasing $m$ so that $m\geq 6$ will have the same result.
	\label{eg:pawX}
\end{example}

\begin{figure}[!h]
	\begin{center}
		\includegraphics{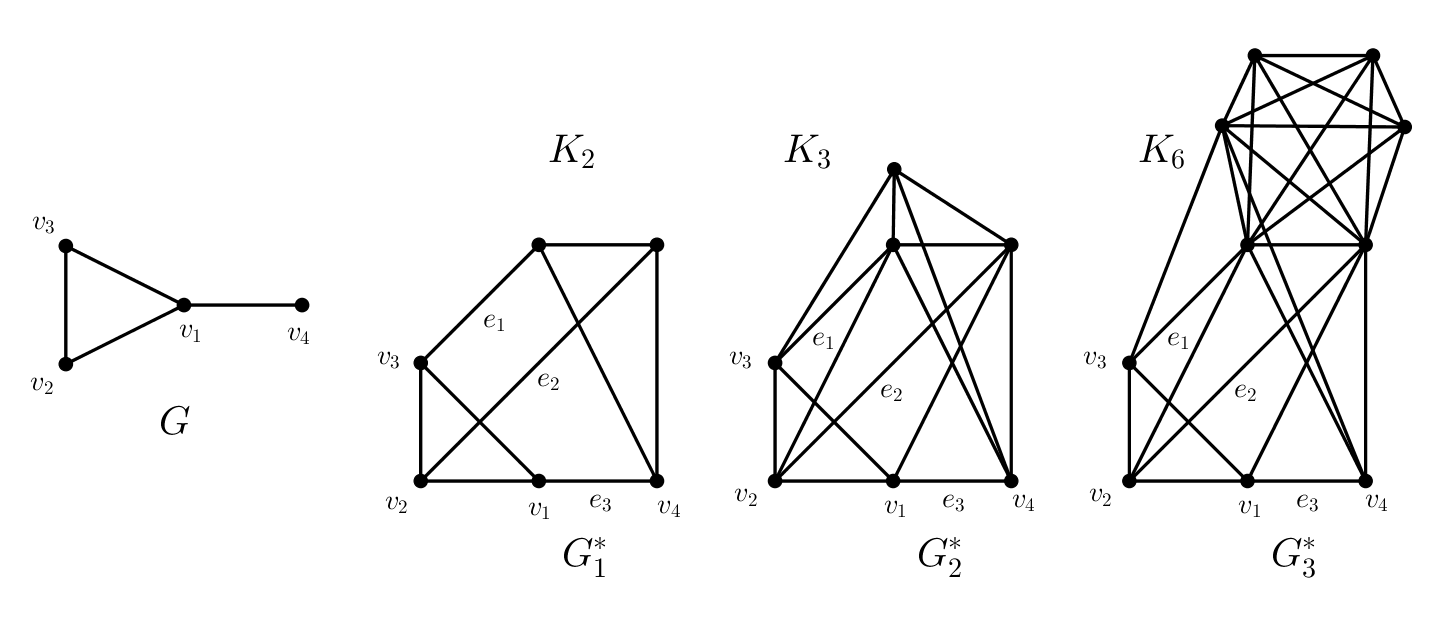}
	\end{center}
	\caption{The ``Paw'' is the mincut graph of a family of graphs $G^*_i$ for $i\geq 3$.}
	\label{fig:pawX}
\end{figure}

\begin{corollary}
	For any connected graph $G$ there is a non-isomorphic graph $H$ such that $X(G)\cong X(H)$.
	\label{corol:Xnoniso}
\end{corollary}

\section{Roots and depth}
\label{sec:rootsdepth}

In this section, we start to explore some further ideas around the mincut operator. Since we have shown that every graph is a mincut graph we have some immediate consequences for the \emph{roots} and \emph{depth}, see \cite{grdynamics, prissurv} of any graph under $X(\cdot)$.

\begin{definition}
	For a given graph operator $\phi$, a $\phi$-\emph{root} of a graph $G$ is any graph $H$ with $\phi(H)\cong G$.
	\label{def:root}
\end{definition}
 
With respect to roots and iteration we also have the following definition of the depth of a graph under some operator.

\begin{definition}
	For a graph $G$ and an operator $\phi$ we define the $\phi$-depth, $depth(G)$, as the largest integer (if there is one, otherwise $\infty$) for which there is some graph $H$ such that $\phi^d(H)\cong G$.
	\label{def:depth}
\end{definition}

\begin{proposition}
	Every graph $G$ has an infinite number of $X$-roots.
	\label{prop:infXroots}
\end{proposition}

\begin{proof}
	The proof follows directly from the construction in Theorem \ref{th:AllGXG}. Since our construction starts with $G\cup K_m$ with appropriate choice of $m$, it follows than any choice of $K_n$ with $n>m$ will have the desired mincut graph.
\end{proof}

\begin{proposition}
	Every graph $G$ has infinite $X$-depth.
	\label{prop:infXdepth}
\end{proposition}

\begin{proof}
	By Theorem \ref{th:AllGXG}, every graph is the mincut graph of a family of graphs. Therefore, given a graph $G_0$, there is some graph $G_{(-1)}$ such that $X(G_{(-1)})\cong G_0$, where the value of the subscript is merely to indicate position in a sequence. By the same reasoning, there is a graph $G_{(-2)}$ such that $X(G_{(-2)})\cong G_{(-1)}$. Thus we can construct a set $\{\ldots,\, G_{(-3)}, \, G_{(-2)}, \, G_{(-1)}, \, G_0\}$ where every $X(G_{(-n)})\cong G_{(-n+1)}$ for all $n\geq 1$ and $X^n(G_{(-n)})\cong G_0$.
\end{proof}

\begin{example}
	We refer to the graphs in Figure \ref{fig:Xiter}. Each sequence represents a number of iterations under $X(\cdot)$. We know that $X(P_4)\cong \overline{K_3}$ since it is basically a tree on four vertices, and since $\overline{K_3}$ is disconnected $X(\overline{K_3})\cong K_0$ by Lemma \ref{lem:nullgraph}. The graph $G_1$ was obtained from $P_4$ by the construction in Theorem \ref{th:AllGXG} using $P_4\cup K_2$. We note that any graph with exactly three mincuts that do not intersect is also a root of $\overline{K_3}$ and any disconnected graph is a root of $K_0$. In the second iteration sequence we have $C_4$ as a root of $L(K_4)$. We could have constructed a root using our construction method from Theorem \ref{th:AllGXG} or, since line graphs of complete graphs are fixed under $X(\cdot)$ we could simply have $L(K_4)$. $W_5$ is a root of $C_4$ since $W_5\cong K_1\vee C_4$ and $C_4$ is $2$-regular with $n=4>2+1$. We could apply our construction technique again to $G_1$ and $W_5$ respectively to continue each sequence in reverse and generate an infinite sequence of graphs such that each is an $X$-root of the next. Admittedly our graphs will become very large very soon but it is possible in principle.
	\label{eg:rootsanddepth}
	
	\begin{figure}[!h]
		\begin{center}
			\includegraphics{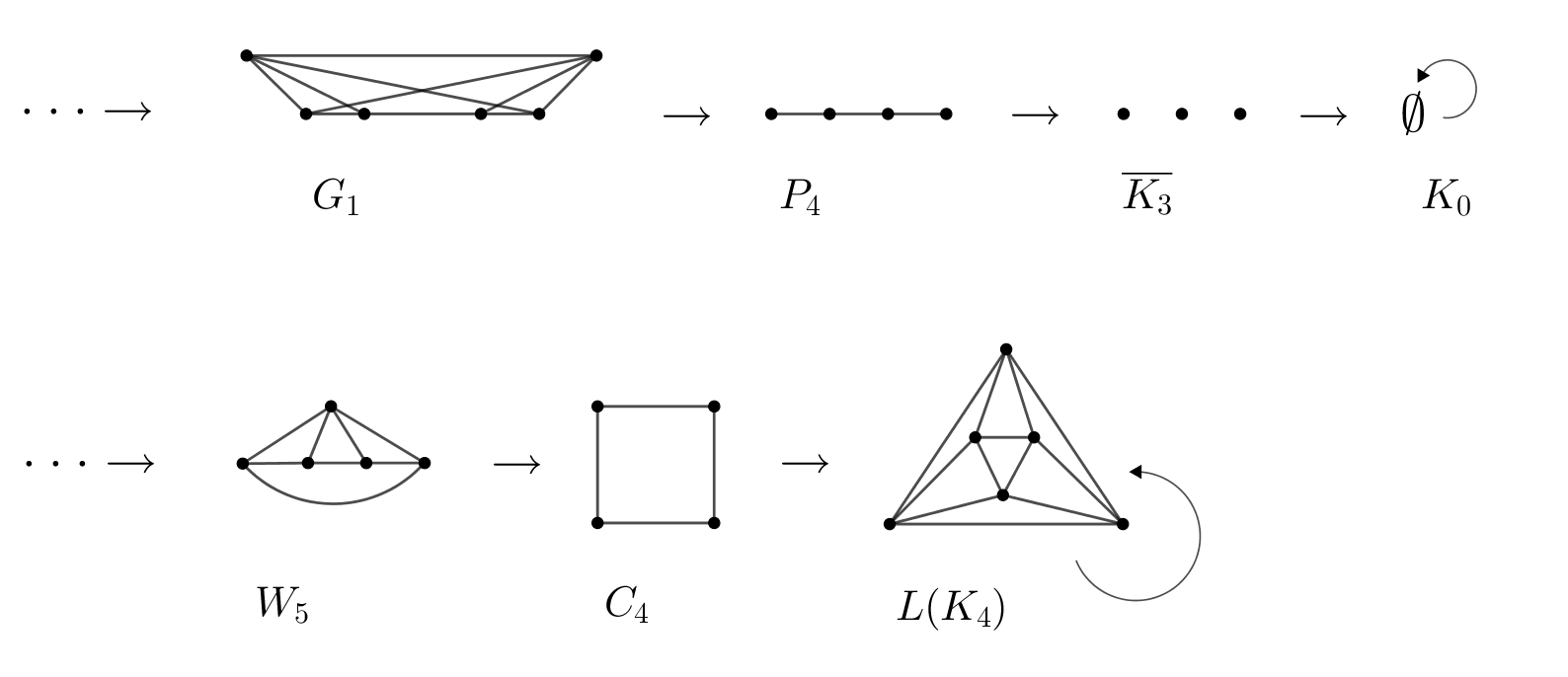}
		\end{center}
		\caption{Graphs with some of their $X$-roots.}
		\label{fig:Xiter}
	\end{figure}
	
\end{example}

\section{Further discussion}

We conclude by introducing and exploring further topics and questions raised by the properties and characteristics of the mincut graph. We also look at the possible implications of certain induced subgraphs in $X(G)$ for a root if that root is not necessarily obtained from the construction process in Theorem \ref{th:AllGXG}.

\subsection{Determination problem}

We know from Whitney that a graph is uniquely determined by its line graph up to isolated vertices, unless it is $C_3$ or $K_{1,3}$, see \cite{prissurv}. Although our construction in Theorem \ref{th:AllGXG} gives the interesting result that every graph is a mincut graph, it unfortunately does not shed much light on what the mincut graph $X(G)$ tells us about the mincuts of $G$ if $G$ is not a graph of the family obtained by our construction. We explore this problem in a little more detail.

\begin{lemma}
	Let $X=\langle A, \,\bar{A}\rangle$ and $Y=\langle B,\, \bar{B}\rangle$ be two non-trivial mincuts of a graph $G$ with $A,\, \bar{A}$ and $B, \, \bar{B}$ the vertex sets of the components of $G-X$ and $G-Y$ respectively. If $X\cap Y \neq \emptyset$ then either $A\cap B\neq \emptyset$ or $A\cap \bar {B}\neq \emptyset$.
	\label{lem:XYcomp}
\end{lemma}

We note that the converse to Lemma \ref{lem:XYcomp} is not necessarily true.

We know, see \cite{chandran,lehel}, that if $X=\langle A, \,\bar{A}\rangle$ and $Y=\langle B,\, \bar{B}\rangle$ are two mincuts of $G$ and $A\cap B\neq \emptyset$ then either $A\subset B$ (or $B\subset A$), that is they are \emph{nested}, or they \emph{overlap}. If $X$ and $Y$ overlap (also called \emph{crossing mincuts}), then $A\cap B$, $\bar{A}\cap B$ and $A\cap \bar{B}$ are \emph{non-empty}. We also note from \cite{lehel} that two mincuts can overlap only if $\lambda$, the minimum edge connectivity number of the graph, is even.

\begin{proposition}[Lehel et al, \cite{lehel}]
	Let $X=\langle A,\bar{A} \rangle$ and $Y=\langle B,\bar{B} \rangle$ be two overlapping mincuts of a connected graph $G$. Then $|\langle \bar{A}\cap \bar{B}, A\cap \bar{B} \rangle |=|\langle \bar{A}\cap \bar{B}, \bar{A}\cap B \rangle |=|\langle A\cap B, A\cap \bar{B} \rangle |=|\langle A\cap B, \bar{A}\cap B \rangle |=\dfrac{\lambda}{2}$. Consequently $A\cup B, \, A\cap B, \, A\cap \bar{B}, \, \bar{A}\cap B$ are mincuts. Moreover, $|\langle \bar{A}\cap \bar{B}, A\cap B \rangle |=|\langle \bar{A}\cap B, A\cap \bar{B} \rangle |=0$.
	\label{prop:crossing}
\end{proposition}

\begin{lemma}[Chandran \& Ram, \cite{chandran}]
	If $X=\langle A,\bar{A} \rangle$ and $Y=\langle B,\bar{B} \rangle$ are a pair of crossing mincuts, then $X\cap Y = \emptyset$.
	\label{lem:crossempty}
\end{lemma}

Let $X=\langle A,\, \bar{A}\rangle$ and $Y=\langle B,\, \bar{B}\rangle$ be overlapping mincuts of a graph $G$ with non-trivial mincuts $W=\langle A\cap B,\, \bar{A} \cup \bar{B} \rangle$, $Z=\langle A\cup B,\bar{A}\cap \bar{B} \rangle$ as per Proposition \ref{prop:crossing}. Suppose that $W\cap X,\,Y$ and $X,\, Y\cap Z$ are non-empty, then we have a corresponding induced cycle $C_4$ in $X(G)$. Given what we know from Proposition \ref{prop:crossing}, it would seem that we should at least be able to construct an induced subgraph in $G$ corresponding to the cycle in $X(G)$ without resorting to the construction in Theorem \ref{th:AllGXG}.

Similarly, let $X=\langle A,\, \bar{A} \rangle$, $Y=\langle B,\, \bar{B} \rangle$ and $Z=\langle C,\, \bar{C}\rangle$ be non-trivial nested mincuts of a graph $G$ with $A\subset B \subset C$. If the induced subgraph in $X(G)$ corresponding to the mincuts $X,\, Y$ and $Z$ is connected then we should have a cycle $C_3$ or at least an induced path $P_3$ in $X(G)$.

It would be interesting to know whether we can have induced cycles larger than $C_4$ in $X(G)$ such that the corresponding mincuts in $G$ are not of the form as per our construction in Theorem \ref{th:AllGXG}.

\subsection{Periodicity of the mincut operator}

\begin{definition}
	Let $\phi$ be an operator on a class of graphs $\Gamma$. A graph $G\in \Gamma$ is \emph{periodic} in $\phi$ if there is some natural number $n$ with $G \cong \phi^n(G)$. The smallest such number is called the \emph{period} of $G$.
	\label{def:periodicity}	
\end{definition}

In Corollary \ref{corol:XGisGifsuper-r}, we showed that it is sufficient for $G\cong X(G)$ if $G$ is super-$\lambda$ and $r$-regular. That is, the graph is fixed under the operator, or has periodicity $1$. In fact, it can be shown that this condition is also necessary.

Furthermore, we know that there are $2$-periodic graphs under $X(\cdot)$. For example, the cartesian product $K_m\Box K_2$ has mincut graph $X(K_m \Box K_2)\cong K_1\vee (K_m\Box K_2)$, the join of itself with $K_1$. But the mincut graph of $K_1\vee (K_m\Box K_2)$ is again $K_m\Box K_2$ which means that iteration of the $X(\cdot)$ operator leads to a cycle of two graphs. A natural question to ask is whether there are other graphs of this periodicity and if so what would be their characteristics. Furthermore, do we have graphs of a higher $X$-periodicity?

\subsection{Convergence or divergence}

In \cite{interchangegr}, Wilf and Van Rooij showed that unless a graph is $K_{1,3}$, $C_n$ or $P_n$, repeated application of the line graph operator eventually leads to a steadily increasing number of vertices, that is, it \emph{diverges} under $L(\cdot)$.

We know that the maximum number of mincuts in a graph of order $n$ is $\binom{n}{2}$ and this bound is tight for simple graphs when $G\cong C_n$, see \cite{dinic,lehel}. We also have other upper bounds on the number of mincuts of a graph depending on whether $\lambda$ is even or odd as mentioned in Section \ref{sec:intro}.

Intuitively, if a graph is to diverge under $X(\cdot)$ both the number of vertices and edges need to increase. If only the number of vertices increases the edge density of the graph decreases and the likelihood of every mincut intersecting with at least one other decreases, converging to the null graph in the end. If only the number of edges increases we eventually end up with a complete graph which is fixed. But for $|V(X(G))|$ and $|E(X(G))|$ to increase and $X(G)$ to remain connected we must have the edge density increasing. Thus we would expect such a graph to eventually become fixed (or at least periodic) under the $X$-operator.

Hence, given the upper bounds on the number of mincuts, that we expect graphs that ``grow'' to eventually become fixed or periodic, that non-intersecting mincuts are disconnected and that the mincut graph of a disconnected graph is the null graph, $K_0$, we have the following two conjectures:

\begin{conjecture}[Convergence]
	Let $G$ be a simple connected graph and $X(\cdot)$ the mincut operator. Then $X^n(G)$ converges for sufficiently large $n$.\\
\end{conjecture}

\begin{conjecture}[Convergence to null graph]
	Let $G$ be a simple connected graph and $X(\cdot)$ the mincut operator. Then $X^k(G)\rightarrow K_0$ for sufficiently large $k$, except under a finite number of conditions.\\
\end{conjecture}

\subsection{Connectivity of $X(G)$}

$X(G)$ will be connected if every mincut has non-empty intersection with at least one other mincut. What are the characteristics of a connected graph that will imply that its mincut graph is connected? Is it possible to find minimum bounds on $\kappa(G), \, \lambda(G),$ and $\delta(G)$, (the vertex connectivity, edge connectivity and minimum degree values of $G$) that will guarantee a connected $X(G)$? Intuitively there should be some relationship between the vertices and the way the edges are distributed (some kind of ``edge density'' function?) and $\lambda(G)$ that will ensure connectivity of $X(G)$.

\bibliographystyle{abbrv}
\bibliography{cutset}
		
\end{document}